\documentclass[11pt]{article}

\usepackage[T1]{fontenc}
\usepackage[utf8]{inputenc}
\usepackage{lmodern}
\usepackage[margin=1.05in]{geometry}
\usepackage{microtype}
\usepackage{amsmath,amssymb,amsthm,mathtools}
\usepackage{enumitem}
\usepackage[dvipsnames]{xcolor}
\usepackage{aliascnt}
\usepackage{hyperref}
\usepackage[nameinlink,noabbrev]{cleveref}
\usepackage{fancyhdr}

\hypersetup{
  colorlinks=true,
  linkcolor=MidnightBlue,
  citecolor=ForestGreen,
  urlcolor=BrickRed,
  pdftitle={Recognizability equals CMSO-definability for graphs of rank-width at most two},
  pdfauthor={Antonios Kalampakas},
  pdfsubject={VR-recognizability and CMSO-definability on graphs of rank-width at most two},
  pdfkeywords={rank-width, VR-recognizability, CMSO, split decomposition, connectivity functions}
}

\newtheorem{theorem}{Theorem}[section]
\newaliascnt{lemma}{theorem}
\newtheorem{lemma}[lemma]{Lemma}
\aliascntresetthe{lemma}
\newaliascnt{proposition}{theorem}
\newtheorem{proposition}[proposition]{Proposition}
\aliascntresetthe{proposition}
\newaliascnt{corollary}{theorem}
\newtheorem{corollary}[corollary]{Corollary}
\aliascntresetthe{corollary}
\newaliascnt{claim}{theorem}

\aliascntresetthe{claim}
\theoremstyle{definition}
\newaliascnt{definition}{theorem}
\newtheorem{definition}[definition]{Definition}
\aliascntresetthe{definition}
\newaliascnt{construction}{theorem}
\newtheorem{construction}[construction]{Construction}
\aliascntresetthe{construction}
\theoremstyle{remark}
\newaliascnt{remark}{theorem}
\newtheorem{remark}[remark]{Remark}
\aliascntresetthe{remark}

\newcommand{\F}{\mathbb F}
\newcommand{\CMSO}{\mathrm{CMSO}}
\newcommand{\MSO}{\mathrm{MSO}}
\newcommand{\fcl}{\operatorname{fcl}}
\newcommand{\rw}{\operatorname{rw}}
\newcommand{\lrw}{\operatorname{lrw}}
\newcommand{\lcw}{\operatorname{lcw}}
\newcommand{\pw}{\operatorname{pw}}

\title{\bfseries Recognizability equals CMSO-definability\\
for graphs of rank-width at most two}
\author{Antonios Kalampakas\\[-0.1em]
\small Department of Mathematics, American University of the Middle East\\[-0.1em]
\small Egaila 54200, Kuwait}
\date{12 July 2026}

\begin{document}
\maketitle

\begin{abstract}
We prove that, on finite graphs of rank-width at most two,
VR-recognizability and counting monadic second-order definability coincide.
This advances the recognizability-versus-definability problem from bounded
linear clique-width to the first nontrivial bounded rank-width level beyond
the rank-width-one split-decomposition case.  The proof first treats
split-prime graphs.  The maximal partial-tree theory of Clark and Whittle
organizes the non-sequential cut-rank-two separations, while a single strong
separation orients all strong equivalence classes and yields a
CMSO-definable laminar family of canonical cores.  Although the auxiliary
partial tree is not itself transduced, it proves that every canonical local
piece has a port-contiguous layout of uniformly bounded linear rank-width.
The width argument uses partition atoms and the branch-width-three display
theorem of Hall, Oxley, Semple, and Whittle and does not assume that graph
torsos remain prime.  Coherent ordered rank-two frames then permit a
finite-state bottom-up evaluation whose local transitions are definable by
the bounded-linear-clique-width theorem of Boja\'nczyk, Grohe, and
Pilipczuk.  Finally, the CMSO-transducible canonical split decomposition
lifts the result from prime graphs to arbitrary graphs of rank-width at most
two.
\end{abstract}



\section{Introduction and main result}

Courcelle's programme relates logical descriptions of graph properties to
finite-state evaluation on algebraic graph decompositions.  For sparse
graphs, the resulting equivalence between recognizability and counting
monadic second-order definability on every bounded-tree-width class was
proved by Boja\'nczyk and Pilipczuk~\cite{BP16}.  The corresponding dense
setting is governed by clique-width and rank-width.  Clique-width was
developed by Courcelle and Olariu~\cite{CourcelleOlariu00}, while rank-width,
introduced by Oum and Seymour~\cite{Oum05,OumSeymour06}, measures the
\(\F_2\)-rank of the cuts in a branch decomposition and is equivalent to
clique-width up to an exponential change of parameter; see also the survey
of Oum~\cite{Oum17}.

For vertex-replacement, or VR, recognizability, the forward logical
implication is classical: every \(\CMSO_1\)-definable graph property is
VR-recognizable~\cite{Courcelle90,CE12}.  The converse fails without a width
restriction, and the central question is therefore whether it becomes true
on bounded-clique-width, equivalently bounded-rank-width, classes.  A major
positive result of Boja\'nczyk, Grohe, and Pilipczuk proves the equivalence
on every class of bounded linear clique-width~\cite{BGP21}.  That theorem
does not settle bounded rank-width: even rank-width-one graphs have
unbounded linear rank-width, as is already witnessed by trees
\cite{AdlerKante15}.  Recent CMSO transductions for canonical split
decompositions, combined with the standard finite-state evaluation of their
bounded-width outputs, nevertheless provide the rank-width-one case
\cite{CampbellEtAl26}.  The present paper establishes the next level,
rank-width at most two.

We fix the conventions used throughout.  All graphs are finite, simple, and
undirected, and \(\CMSO\) means \(\CMSO_1\): monadic second-order logic over
the vertex-adjacency structure, extended by all fixed modular-cardinality
predicates.  We use the usual many-sorted VR algebra, whose sorts are
indexed by finite vertex-label alphabets and whose operations are disjoint
union, relabelling, and complete joining of label classes.  Labels form a
partition of the vertex set, and an unlabelled graph is represented in the
one-label sort.  A congruence is \emph{locally finite} if its restriction to
each sort has finite index.  A graph property \(\mathcal P\) is
\emph{globally VR-recognizable} if it is recognized by a locally finite
quotient of this algebra, equivalently if a locally finite congruence
saturates \(\mathcal P\) in the one-label sort
\cite{CE12,CourcelleWeil05}.  Relabelling every vertex to that label shows
that the same congruence saturates the inverse image of \(\mathcal P\)
under forgetting labels in every sort.  A one-hole polynomial context may
use arbitrary fixed labelled graphs as operands, and congruence
compatibility permits replacement within such a context.  The fixed
operands need not admit terms over the working label alphabet alone.
We use nondeterministic CMSO transductions in their standard finite-copy
form, with closure under composition and backward translation
\cite{CE12,BGP21}.

\begin{theorem}[Main theorem]\label{thm:main}
Let
\[
  \mathcal C_2=\{G:\rw(G)\leq 2\}.
\]
For every language \(\mathcal L\subseteq\mathcal C_2\), the following are
equivalent:
\begin{enumerate}[label=\textup{(\roman*)}]
  \item there is a globally VR-recognizable graph property \(\mathcal P\)
        such that \(\mathcal L=\mathcal P\cap\mathcal C_2\);
  \item there is a \(\CMSO\) sentence \(\varphi\) such that
        \[
          \mathcal L=\{G\in\mathcal C_2:G\models\varphi\}.
        \]
\end{enumerate}
\end{theorem}

Thus recognizability on \(\mathcal C_2\) means restriction to
\(\mathcal C_2\) of a globally VR-recognizable property; this avoids imposing
any condition on how a relative language is extended outside the class.  The
implication \(\textup{(ii)}\Rightarrow\textup{(i)}\) is Courcelle's
recognizability theorem.  We prove the converse.

The proof separates a structural witness from the logical
object that is ultimately transduced.  On a split-prime input, the
Clark--Whittle flower and maximal-partial-tree theorems first organize the
non-sequential exact cut-rank-two separations.  Their strong equivalence
classes have canonical cores, and one anchor orientation turns these cores
into a CMSO-definable laminar family.  The Clark--Whittle tree is used only
to prove a width theorem for the local pieces of this canonical family; no
formula has to select that auxiliary tree, its flowers, or any particular
linear layout.  The local pieces are replaced by rank-two port torsos, and a
branch-width-three display theorem gives port-contiguous layouts of linear
rank-width at most six.  Two unary selectors choose coherent boundary bases.
For a fixed recognizable property, its finite port-congruence states can
then be verified locally by CMSO because all interpreted local expansions
have bounded linear clique-width.  A finite-state computation on the
laminar tree proves the prime case.  Finally, the canonical split
decomposition of Cunningham~\cite{Cunningham82}, in the CMSO-transducible
form of Campbell et al.~\cite{CampbellEtAl26}, lifts the argument to all
connected graphs, and a finite commutative-monoid calculation handles
connected components.

\section{Connectivity preliminaries}

For a graph \(G\) and \(X\subseteq V(G)\), let
\[
 \rho_G(X)=\operatorname{rank}_{\F_2} A_G[X,V(G)\setminus X].
\]
This function is symmetric, submodular, vanishes at the empty set, and has
\(\rho_G(\{v\})\leq 1\).  We use the same notation \(\lambda\) for an
arbitrary symmetric submodular integer-valued function satisfying these
three conditions.  Branch-width and path-width of \(\lambda\) are defined
in the usual way.  With this convention, graph rank-width is the
branch-width of \(\rho_G\), and graph linear rank-width is its path-width.

Two shifts used in the literature will be kept explicit.  Clark and Whittle
call \(X\) \(k\)-separating when their connectivity function has value at
most \(k\); we apply their theory directly with \(k=2\) and
\(\lambda=\rho_G\)~\cite{ClarkWhittle}.  Hall, Oxley, Semple, and Whittle
normalize a 3-connected connectivity function to have value one at the
empty set.  Accordingly, only in the application of their display theorem
in \cref{lem:atom-compression} do we pass from an atom cut-rank function
\(\mu\) to the shifted function \(\mu^+=\mu+1\).  No other argument changes
the cut-rank convention.

A connected graph with at least five vertices is \emph{split-prime} when it
has no partition \(X\mid V(G)\setminus X\), with both sides of size at least
two, whose cut-rank is at most one.  Thus a split-prime graph satisfies
\(\rho_G(X)\geq2\) whenever both sides have size at least two.  The bounded
graphs excluded by this size convention will always be handled by a finite
disjunction.

\begin{definition}[Absorption closure]
If \(e\notin X\), then \(e\) is \emph{absorbable into} \(X\) when
\(\lambda(X\cup\{e\})\leq\lambda(X)\).  The full absorption closure
\(\fcl(X)\) is the set reached by repeatedly adding absorbable elements.
An exact \(2\)-separation \((X,E\setminus X)\) is \emph{sequential} when
\(\fcl(X)=E\) or \(\fcl(E\setminus X)=E\), and is \emph{non-sequential}
otherwise.  A side of a non-sequential exact \(2\)-separation is abbreviated
an \emph{ns-side}.
\end{definition}

\begin{lemma}[Persistence and uniqueness]\label{lem:persistence}
If \(e\) is absorbable into \(X\), then it remains absorbable into every
\(X\cup W\) not already containing \(e\).  Consequently \(\fcl(X)\) is
well-defined, is the least absorption-closed superset of \(X\), and is
monotone in \(X\).
\end{lemma}

\begin{proof}
Submodularity applied to \(X\cup\{e\}\) and \(X\cup W\) gives
\[
 \lambda(X\cup W\cup\{e\})-\lambda(X\cup W)
 \leq \lambda(X\cup\{e\})-\lambda(X)\leq0.
\]
Thus an available absorption step remains available after any other steps.
All maximal absorption sequences therefore add the same elements.  The
least-closed-superset and monotonicity assertions follow by shadowing an
absorption sequence from the smaller set inside the larger set.
\end{proof}

\begin{lemma}[Small sides]\label{lem:small-sides}
Assume that \(\lambda\) is 3-connected in the unshifted sense:
\(\lambda(Z)\geq2\) whenever \(|Z|,|E\setminus Z|\geq2\).  If
\(\lambda(X)=2\) and \(|X|\leq3\), then
\(\fcl(E\setminus X)=E\).  Hence every ns-side and its complement have at
least four elements.
\end{lemma}

\begin{proof}
Add the elements of \(X\) to its complement one at a time.  The successive
connectivity values are bounded by \(2,2,1,0\), in the appropriate initial
segment, and 3-connectivity forces the two-element stage to have value two.
Thus every addition is an absorption step.
\end{proof}

\begin{lemma}[Reversal]\label{lem:reversal}
If \(\lambda(X)\leq2\) and \(\fcl(X)=E\), then \(E\setminus X\) has an
ordering every prefix of which has connectivity at most two.
\end{lemma}

\begin{proof}
Read an absorption sequence from \(X\) to \(E\) backwards and use symmetry
of \(\lambda\).
\end{proof}

\begin{lemma}[Gluing sequential sets]\label{lem:gluing}
Suppose \(E=W_1\dot\cup W_2\dot\cup W_3\), each \(W_i\) has a prefix
ordering of width at most two, and \(\lambda(W_1),\lambda(W_3)\leq2\).
Then \(\pw(\lambda)\leq4\).
\end{lemma}

\begin{proof}
Order \(W_1\), then \(W_2\), then the reverse of an ordering of \(W_3\).
A prefix in the middle block is \(W_1\cup Z\), with
\(\lambda(W_1\cup Z)\leq\lambda(W_1)+\lambda(Z)\leq4\) by disjoint
subadditivity.  Prefixes in the first and last blocks have value at most two.
\end{proof}

The next form of the usual sink argument is useful because later we apply it
to a partitioned connectivity function whose boundary atoms can have
singleton value two.

\begin{lemma}[Displayed-edge sink lemma]\label{lem:sink}
Let \(\mu\) be a nonnegative integer-valued symmetric submodular function on
an atom set \(\mathcal A\), with \(\mu(\varnothing)=0\),
and let \(D\) be a branch decomposition of width at most two.  Assume:
\begin{enumerate}[label=\textup{(\alph*)}]
  \item \(\mu(\{a\})\leq2\) for every atom \(a\);
  \item for every edge of \(D\), at least one displayed side has full
        absorption closure in \(\mathcal A\).
\end{enumerate}
Then \(\mathcal A\) has a linear ordering of \(\mu\)-width at most four.
\end{lemma}

\begin{proof}
If \(|\mathcal A|\leq2\), any ordering has width at most two.  Assume
\(|\mathcal A|\geq3\).
Orient every leaf edge away from its leaf, since the complementary side
absorbs the remaining atom in one step and therefore has full closure.
Orient every other edge toward any displayed side having full closure.
Hence an internal sink exists.  If
\(W_1,W_2,W_3\) are the atom sets in the three branches at the sink, then
\(\fcl(\mathcal A\setminus W_i)=\mathcal A\).  By
\cref{lem:reversal}, each \(W_i\) has a width-two ordering.  Apply
\cref{lem:gluing}.  Its proof only uses symmetry and submodularity, so the
larger singleton bound causes no difficulty.
\end{proof}

To make this closure available to the later logical construction, we record
its definability at bounded cut-rank.  For each fixed \(r\), the predicate
\(\rho_G(X)\leq r\) is first-order with
one set parameter: guess at most \(r\) rows in \(X\) and require every row in
\(X\) to be one of their \(2^r\) binary combinations on
\(V(G)\setminus X\).  Consequently exact rank two, absorbability over a
rank-at-most-two set, and the property of being absorption-closed are
\(\CMSO\)-definable.

\begin{lemma}[CMSO full closure at width two]\label{lem:logic-closure}
There is a \(\CMSO\) formula \(\operatorname{Fcl}_2(X,Y)\) which, whenever
\(\rho_G(X)\leq2\), holds exactly when \(Y=\fcl(X)\).
\end{lemma}

\begin{proof}
By \cref{lem:persistence}, \(\fcl(X)\) is the least absorption-closed
superset of \(X\).  Say that \(Y\) contains \(X\), has rank at most two, is
closed against every singleton absorption, and is contained in every set
with those three properties.  This is a monadic second-order definition.
Every absorption sequence from a rank-at-most-two set stays at rank at most
two, so the definition returns precisely \(\fcl(X)\).
\end{proof}

\section{Flowers, strong classes, and the canonical laminar tree}

Throughout this section, \((E,\lambda)\) is the cut-rank connectivity system
of a split-prime graph of rank-width at most two.  For at least five vertices,
split-primality says precisely that
\[
 \lambda(X)\geq2\qquad
 \text{whenever }|X|,|E\setminus X|\geq2.
\tag{3.1}\label{eq:prime-connectivity}
\]
If \(|E|\leq8\), we treat the whole graph as one bounded local piece and do
not invoke the Clark--Whittle partial-tree machinery.  Accordingly, every
statement below involving a maximal partial \(2\)-tree is made under the
additional standing assumption \(|E|\geq9\).  The closure formulas and the
strong-core set system remain meaningful at every order.

We now recall the order-two form of the Clark--Whittle machinery.  Let
\(\mathcal T_0\) consist of the empty set and all singleton subsets of
\(E\).  Under \eqref{eq:prime-connectivity}, this is the unique singleton
tangle of order two.  If \(|E|\geq9\), it is robust in the sense of
Clark--Whittle: eight \(\mathcal T_0\)-weak sets, each of size at most one,
cannot cover \(E\).  Moreover, their tangle full closure is exactly the
absorption closure of \cref{lem:persistence}, because the nonempty
\(\mathcal T_0\)-weak sets are singletons.  While both sides contain at
least two elements, every permitted intermediate set has value exactly two
by \eqref{eq:prime-connectivity}, so ``remain 2-separating'' is the same as
``do not increase''; once the complementary side has size at most one, the
remaining values decrease to zero.

Define
\[
 \mathcal S=\{X\subseteq E:\lambda(X)\leq2,\
 E\setminus X\text{ is }\mathcal T_0\text{-strong},\
 \fcl(E\setminus X)\neq E\}.
\]
This is the basic tree-compatible family of Clark--Whittle
\cite[p.~8, axioms (S1)--(S2)]{ClarkWhittle}.  Their
\((2,\mathcal S)\)-separations require both sides to belong to
\(\mathcal S\), so they are exactly the non-sequential exact
\(2\)-separations used here.  Equality of the unordered closure pairs
therefore gives axiom (S1).  For (S2), if \(X\in\mathcal S\) and
\((Y,E\setminus Y)\) is a \(\mathcal T_0\)-strong \(2\)-separation with
\(X\subseteq Y\), then monotonicity gives
\[
 \fcl(E\setminus Y)\subseteq\fcl(E\setminus X)\neq E,
\]
and hence \(Y\in\mathcal S\).  Membership of \(X\) alone does not require
\(\fcl(X)\neq E\).  Clark--Whittle partial 2-sequences now
consist of singleton additions; their Lemmas~3.4 and~3.6 identify the full
closure with a maximal such sequence, while Lemma~3.8 gives the closure-pair
equivalence calculus used below.  Thus the terminology in this paper is a
literal specialization of theirs rather than a different closure theory.

We use two consequences of that theory.  The first combines the defining
axioms (P1)--(P5) of a partial \((2,\mathcal S)\)-tree, Theorem~7.1, and
Corollary~7.2 of Clark--Whittle~\cite{ClarkWhittle}.  A maximal tree exists
because the finite system has only finitely many separation classes.

\begin{theorem}[Clark--Whittle partial-tree theorem, specialized]
\label{thm:CW}
Let \((E,\lambda)\) satisfy \eqref{eq:prime-connectivity}, with \(|E|\geq9\).
There is an \(\mathcal S\)-maximal partial \((2,\mathcal S)\)-tree \(T\),
abbreviated a maximal partial \(2\)-tree, with bag nodes and flower nodes,
such that every non-sequential exact \(2\)-separation is equivalent, under
equality of the two full closures, to a separation displayed by an edge or a
flower node of \(T\).  Every flower has no loose petal, is either an
anemone or a daisy, and has \(\mathcal S\)-order at least three.  In
particular, it displays at least two inequivalent non-sequential classes.
\end{theorem}

The second consequence is exactly Clark--Whittle Lemma~6.2 combined with
their robust-tangle conformance theorem, Theorem~5.10
\cite{ClarkWhittle}.

\begin{theorem}[Clark--Whittle maximal-flower theorem, specialized]
\label{thm:CW-flower}
Under the hypotheses of \cref{thm:CW}, for every non-sequential exact
\(2\)-separation there is an \(\mathcal S\)-tight
\(\mathcal S\)-maximal \(2\)-flower \(\Phi\) which displays an equivalent
separation.  Every non-sequential exact \(2\)-separation conforms with
\(\Phi\): after replacement by an equivalent separation, it is displayed
as a union of petals or one of its sides is contained in a petal.
\end{theorem}

Here an anemone permits every union of petals, whereas a daisy permits the
cyclically consecutive unions.  We shall use only the display, conformance,
no-loose-petal, and terminal-bag relocation conclusions of these results.

Two ns-sides are \emph{equivalent} if their unordered closure pairs agree:
\[
 \bigl\{\fcl(X),\fcl(E\setminus X)\bigr\}
 =
 \bigl\{\fcl(Y),\fcl(E\setminus Y)\bigr\}.
\]
Two equivalence classes are \emph{compatible} if they have representatives
whose bipartitions do not cross.  A class is \emph{strong} if it is compatible
with every non-sequential class.

\begin{lemma}[The narrow flower lemma]\label{lem:narrow-flower}
If \((E,\lambda)\) has branch-width at most two and has no strong
non-sequential class, then \(\pw(\lambda)\leq4\).
\end{lemma}

\begin{proof}
If \(|E|\leq8\), then
\(\lambda(X)\leq\min\{|X|,|E\setminus X|\}\leq4\) for every \(X\), so any
ordering has width at most four.  Hence assume \(|E|\geq9\).
If no non-sequential separation exists, orient the edges of a width-two
branch decomposition toward a full-closure side.  Small displayed sides are
handled by \cref{lem:small-sides}; an internal sink and
\cref{lem:reversal,lem:gluing} give width four.

Otherwise choose a non-sequential exact \(2\)-separation.  By
\cref{thm:CW-flower}, there is an \(\mathcal S\)-tight
\(\mathcal S\)-maximal flower
\(\Phi=(P_1,\ldots,P_m)\) which displays an equivalent separation and with
which every non-sequential class conforms.  Tightness and Clark--Whittle
Lemma~4.3 imply that \(\Phi\) has no loose petal.

No-looseness implies \(\fcl(P_i)\neq E\): if \(\fcl(P_i)=E\), then a petal
adjacent to \(P_i\) in the flower order is contained in \(\fcl(P_i)\) and is
therefore loose.
Suppose also that \(\fcl(E\setminus P_i)\neq E\).  Then the petal cut is a
non-sequential class.  By \cref{thm:CW-flower}, every other non-sequential
class has an equivalent representative which is either a union of petals or
has one side contained in one petal.  Such a representative is noncrossing
with \(P_i\mid E\setminus P_i\).  Hence the petal class is strong, a
contradiction.  Therefore
\[
 \fcl(E\setminus P_i)=E\qquad(i=1,\ldots,m).
\tag{3.2}\label{eq:outward}
\]
By \cref{lem:reversal}, every petal has a width-two prefix ordering.
Concatenate these orderings in the cyclic flower order.  A prefix meeting
\(P_i\) is \(U\cup Z\), where \(U\) is a consecutive union of earlier petals
and \(Z\) is a prefix inside \(P_i\).  The flower axioms and disjoint
subadditivity give
\[
 \lambda(U\cup Z)\leq\lambda(U)+\lambda(Z)\leq2+2=4.
\]
\end{proof}

The strong equivalence classes are represented by canonical cores, which we
now define.  For an oriented ns-side write
\[
 \vec X=(A_X,B_X):=(\fcl(X),\fcl(E\setminus X)),
 \qquad \vec X^{\,*}=(B_X,A_X),
\]
and put
\[
 (A,B)\preceq(C,D)\quad\Longleftrightarrow\quad
 A\subseteq C\text{ and }B\supseteq D.
\tag{3.3}\label{eq:closure-order}
\]
The \emph{core} of this orientation is
\[
 K(A,B)=E\setminus B.
\tag{3.4}\label{eq:core}
\]
The closure calculus of Clark--Whittle implies that \(K(A_X,B_X)\) is an
ns-side equivalent to \(X\), specifically by
\cite[Lemma~3.8(iii)--(iv)]{ClarkWhittle}.  Concretely, \(B_X\) is fully
closed, symmetry makes its complement exactly 2-separating, and the
partial-sequence argument moves the fringe
\(X\setminus K(A_X,B_X)\) without changing either full closure.

The following elementary form of that calculus will be used repeatedly.

\begin{lemma}[Closure-pair compatibility]\label{lem:pair-compatibility}
Two non-sequential classes are compatible if and only if some orientations of
their closure pairs are comparable under \(\preceq\).
\end{lemma}

\begin{proof}
If \((A,B)\preceq(C,D)\), then the corresponding cores satisfy
\(E\setminus B\subseteq E\setminus D\), and give nested representatives.
Conversely, replace two noncrossing representatives by their cores and orient
the nonempty corner toward the containment.  Monotonicity of \(\fcl\) gives
the two containments in \eqref{eq:closure-order}.  This is also the
order-two case of the Clark--Whittle equivalence calculus.
\end{proof}

\begin{lemma}[Nested equivalent sides]\label{lem:nested-equivalent}
Let \(R\subseteq S\) be oriented equivalent ns-sides.  Then the elements of
\(S\setminus R\) can be ordered \(s_1,\ldots,s_t\) so that
\[
 \lambda(R\cup\{s_1,\ldots,s_j\})=2
 \qquad(0\leq j\leq t).
\tag{3.5}\label{eq:nested-sequence}
\]
\end{lemma}

\begin{proof}
Take an absorption sequence from \(R\) to \(\fcl(R)\).  All its proper
prefixes have value exactly two: they contain the ns-side \(R\), while their
complements contain \(E\setminus\fcl(R)\), and both have at least two
elements.  List the members of \(S\setminus R\) in the order in which they
occur in this sequence.  If \(C\) is the absorption prefix ending at
\(s_j\), then
\[
 R\cup\{s_1,\ldots,s_j\}=S\cap C.
\]
The union \(S\cup C\) contains \(R\), and its complement contains
\(E\setminus\fcl(R)\), so 3-connectivity gives
\(\lambda(S\cup C)\geq2\).  Submodularity now yields
\[
 4=\lambda(S)+\lambda(C)
 \geq\lambda(S\cap C)+\lambda(S\cup C),
\]
so \(\lambda(S\cap C)\leq2\); 3-connectivity makes it equal to two.
\end{proof}

Assume first that a strong class exists and guess one oriented strong ns-side
\(R\).  Put \(r=\vec R\).  For every strong class other than the anchor there
is a unique orientation \(s\) satisfying
\[
 s\preceq r\quad\text{or}\quad s\preceq r^*.
\tag{3.6}\label{eq:anchor-rule}
\]
Existence follows from compatibility with the anchor and
\cref{lem:pair-compatibility}: if \(r\preceq s\), then
\(s^*\preceq r^*\), and similarly for the other reverse comparison.
For uniqueness, if both orientations of \(s=(C,D)\) satisfy
\eqref{eq:anchor-rule}, then either \(C\cup D\) is contained in one proper
anchor closure, which is impossible because \(C\cup D=E\), or
\((C,D)\) is exactly \(r\) or \(r^*\).  We fix the anchor orientation as
\(r\).

\begin{lemma}[Anchor-core laminarity]\label{lem:anchor-laminar}
The canonical cores selected by \eqref{eq:anchor-rule} form a laminar family.
\end{lemma}

\begin{proof}
Let \(s=(C,D)\) and \(t=(F,H)\) be selected orientations.  If they lie on
opposite anchor branches, then
\[
 K(s)\subseteq E\setminus B_R,
 \qquad K(t)\subseteq E\setminus A_R,
\]
which are disjoint because \(A_R\cup B_R=E\).  If they lie on the same
branch, strongness makes their classes compatible.  A comparison
\(s\preceq t\) or \(t\preceq s\) nests their cores.  A facing comparison,
say \(s\preceq t^*\), gives \(D\supseteq F\), hence
\(K(s)\subseteq E\setminus F\), while \(K(t)\subseteq F\), so the cores are
disjoint.  The remaining comparison \(s^*\preceq t\) is impossible.
Indeed, if both selected orientations lie below the anchor orientation
\((A,B)\), then \(C,F\subseteq A\).  That comparison would also give
\(D\subseteq F\), forcing \(E=C\cup D\subseteq A\), contrary to
non-sequentiality of the anchor.  The same argument applies to its opposite
orientation, so the listed nesting and disjointness cases are exhaustive.
\end{proof}

Every ingredient in the preceding construction is \(\CMSO\)-definable.
Indeed, ns-sides and their equivalence are definable by
\cref{lem:logic-closure}; compatibility says that equivalent representatives
exist which do not cross; and strongness universally quantifies over
ns-sides.  With the one guessed anchor set \(R\), define the set predicate
\[
\begin{split}
 \operatorname{SET}_R(Z)\ \Longleftrightarrow\ &
 Z=E\ \vee\ |Z|=1\ \vee\\
 &\exists X\,[\operatorname{StrongNS}(X)\wedge
 \operatorname{Or}_R(X)\wedge
 Z=E\setminus\fcl(E\setminus X)].
\end{split}
\tag{3.7}\label{eq:set-predicate}
\]
Multiple representatives of one class define the same core, so this is
extensional.  If there is no strong class, use only \(E\) and the singletons.
Thus \eqref{eq:set-predicate} is a laminar set system in the convention of
Campbell--Guillon--Kant\'e--Kim--K\"ohler, in particular it does not contain
the empty set.

\begin{corollary}[Definable strong-core tree]\label{cor:strong-tree}
There is a nondeterministic \(\CMSO\)-transduction which, on every prime graph
of rank-width at most two, outputs the rooted laminar tree of the canonical
strong cores, with the original vertices as its leaves.
\end{corollary}

\begin{proof}
Apply Theorem~3.1 of Campbell--Guillon--Kant\'e--Kim--K\"ohler
\cite{CampbellEtAl26} to the definable laminar set system
\eqref{eq:set-predicate}.  Only one anchor set is needed to orient the
classes.  The laminar-tree transduction uses finitely many additional unary
colours, and no orientation set is chosen separately for each class.
\end{proof}

We next relate the canonical core system to the partial tree in
\cref{thm:CW}, thereby establishing the structural interface used later.

\begin{lemma}[No-loose concatenation]\label{lem:no-loose-concatenation}
Let \(\Phi=(P_1,\ldots,P_n)\) be a flower of a partial
\((2,\mathcal S)\)-tree, and let
\(U=\bigcup_{i\in I}P_i\) display a non-sequential class, where both \(I\)
and its complement contain at least two petals.  If \(I\) is consecutive in
the flower order, then both concatenations
\[
 (P_i:i\in I,\ E\setminus U)
 \quad\text{and}\quad
 (U,\ P_i:i\notin I)
\]
have no loose petals.
\end{lemma}

\begin{proof}
We record the uncrossing step rather than importing the tightness hypothesis
from Clark--Whittle Lemma~4.7.  Let \(W\) be a consecutive union of at least
two petals whose displayed cut is non-sequential, let \(P\) be a petal
outside \(W\) adjacent to it, and let \(R\subseteq W\) be the boundary petal
adjacent to \(P\) in the original flower.  We claim that
\(P\not\subseteq\fcl(W)\).  Otherwise choose a maximal partial 2-sequence
\((X_h)_{h=1}^m\) for \(W\).  For every \(\ell\), the two 2-separating sets
\[
 W\cup\bigcup_{h\leq\ell}X_h
 \quad\text{and}\quad R\cup P
\]
have a union whose one side contains the ns-side \(W\), while the other
contains \(E\setminus\fcl(W)\).  Both are strong, so that union has
connectivity at least two.  Submodularity therefore makes
\[
 R\cup\left(P\cap\bigcup_{h\leq\ell}X_h\right)
\]
2-separating.  After empty intersections are removed, the sets
\(P\cap X_h\) form a partial 2-sequence for \(R\).  Since
\(P\subseteq\bigcup_hX_h\), this gives \(P\subseteq\fcl(R)\), making \(P\)
a loose petal of the original flower, a contradiction.  This proves the
claim.

Relabel cyclically so that \(I\) is an interval.  In either stated
concatenation, an old petal can become loose only into the closure of the
new merged petal; the claim excludes this.  Conversely, if the merged petal
were contained in the closure of an adjacent old petal, then the original
boundary petal contained in the merged petal would already be loose.  All
other adjacent pairs are unchanged.  Hence neither concatenation has a
loose petal.
\end{proof}

\begin{proposition}[Class-normal form]\label{prop:strong-normal}
Assume \(|E|\geq9\).  A maximal partial \(2\)-tree may be chosen, within its Clark--Whittle
equivalence, so that every strong non-sequential class is displayed by an
edge.  On deleting those edges:
\begin{enumerate}[label=\textup{(\roman*)}]
  \item the components are naturally the inner nodes of the laminar tree in
        \cref{cor:strong-tree};
  \item each component contains at most one flower node;
  \item in a flower component, every petal is either an incident strong
        boundary or satisfies \(\fcl(E\setminus P)=E\), and an
        outward-sequential petal contains no deeper strong boundary;
  \item if a component contains no flower and an ns-side keeps every incident
        boundary core wholly on one side (call such a side
        \emph{canonical-block-saturated}), then its class is one of the
        incident boundary classes.
\end{enumerate}
\end{proposition}

\begin{proof}
We first prove a local replacement claim.  Let \(v\) be a flower vertex with flower
\(\Phi=(P_1,\ldots,P_n)\), and suppose that a strong class
\(\mathfrak s\), not already displayed by an edge, is represented at \(v\) by
the displayed union
\[
 U=\bigcup_{i\in I}P_i.
\]
Both \(I\) and its complement contain at least two petals, since otherwise
the corresponding incident edge already displays \(\mathfrak s\).  In an
anemone reorder the petals so that \(I\) is consecutive; in a daisy it is
consecutive because \(U\) is displayed.

Replace \(v\) by two adjacent vertices \(v_U,v_{\bar U}\).  Attach the
\(U\)-branches to \(v_U\), the remaining branches to \(v_{\bar U}\), and let
the new edge display \(U\mid E\setminus U\).  The partitions at the two new
vertices are
\[
 \Phi_U=(P_i:i\in I,\ E\setminus U),
 \qquad
 \Phi_{\bar U}=(U,\ P_i:i\notin I),
\tag{3.8}\label{eq:flower-split}
\]
in the inherited cyclic orders.  They are concatenations of \(\Phi\), and
hence are \(2\)-flowers.  Each new centre is labelled as an anemone or a
daisy according to the actual type of its displayed flower.

By \cref{lem:no-loose-concatenation}, both \(\Phi_U\) and
\(\Phi_{\bar U}\) have no loose petals.

We next check that no non-sequential class is lost.  Let
\(\mathfrak c\neq\mathfrak s\) have a representative displayed by \(\Phi\).
Since \(\mathfrak s\) is strong, \cref{lem:pair-compatibility} gives an
oriented core representative of \(\mathfrak c\) contained either in \(U\) or
in \(E\setminus U\).  Indeed, if a literal of \(\mathfrak c\) lies below the
\(U\)-literal, its core lies in the \(U\)-core and hence in \(U\).  If the
\(U\)-literal lies below that literal, reverse both orientations; then the
opposite core of \(\mathfrak c\) lies in the opposite \(U\)-core and hence in
\(E\setminus U\).  The two facing cases are identical.  Clark--Whittle
Lemma~4.8 applies with the displayed cuts for \(\mathfrak c\) and
\(\mathfrak s\), and with that core as its equivalent contained
representative.  Its hypotheses hold because the two classes are
inequivalent and the flower has \(\mathcal S\)-order at least three.  It
therefore gives an equivalent representative of \(\mathfrak c\), still
displayed by \(\Phi\), whose corresponding side is contained in that same
set.  A union of old petals contained in \(U\) is displayed by \(\Phi_U\),
and one contained in \(E\setminus U\) is displayed by
\(\Phi_{\bar U}\).  The class \(\mathfrak s\) is displayed by the new
central edge, and displays away from \(v\) are unchanged.  Since the
original maximal partial tree displays an equivalent representative of every
non-sequential class by \cref{thm:CW}, the modified tree retains this stronger
display property for every class.

A new flower has \(\mathcal S\)-order at least two, since it displays
\(\mathfrak s\).  If its order is at least three, it is a legitimate
partial-tree flower vertex.  If its order is two, every non-sequential
separation it displays is equivalent to \(\mathfrak s\).  Replace its centre
by an empty bag vertex, which is permitted in a partial
\((2,\mathcal S)\)-tree.  Whenever this creates a sequential edge between two
bag vertices, contract that edge and merge the two bag labels.  Repeat until
every remaining bag--bag edge is non-sequential.  Edges incident with a
flower vertex still satisfy the partial-tree edge axiom.  Every newly created
bag--bag edge incident with a suppressed centre is either sequential and
contracted, or is non-sequential and represents \(\mathfrak s\);
pre-existing bag--bag edges already satisfy that axiom.  The only
non-sequential class displayed at the suppressed order-two flower was
\(\mathfrak s\), retained by the new central edge, while all attached
subtrees are unchanged.

For completeness, the partial-tree axioms can now be checked directly.
Before an order-two suppression, every new edge incident with a flower
displays a petal cut of a flower and hence a \(\mathcal T_0\)-strong
\(2\)-separation, while the central edge displays the non-sequential class
\(\mathfrak s\); all other edges are unchanged.  This proves (P1), and the
chosen anemone or daisy labels give (P2).  The surviving new flowers have
\(\mathcal S\)-order at least three, have the type recorded by their labels,
and have no
loose petal, proving (P3) and (P4).  On suppressing an order-two flower, its
incident edge separations are unchanged; contracting precisely the sequential
bag--bag edges leaves (P1)--(P4) valid.  Finally, the retained display property
implies (P5), since every non-sequential class has an equivalent separation
displayed by the modified tree.  Thus suppression and contraction preserve
all partial-tree axioms.

The resulting tree displays exactly the same \(\mathcal S\)-equivalence
classes as the original tree.  It is therefore equivalent to it and remains
\(\mathcal S\)-maximal.  The replacement adds an edge displaying
\(\mathfrak s\).  Previously produced strong edges are merely
reattached to one of the two new vertices; none is removed.  Since after every
replacement the tree still displays a representative of every
non-sequential class, the operation can be iterated until every strong class
is displayed by an edge.

We may moreover arrange that every strong class labels exactly one edge.
Suppose two distinct edges display the same non-sequential class.  Choose the
orientations of their displayed sides, call them \(X\) and \(Y\), having the
same first full closure \(A\), and call the common opposite closure \(B\).
Both sides contain the common core \(E\setminus B\), so they meet.  Since edge separations do not
cross, they are nested unless their union is \(E\).  The latter would give
\(B=\fcl(E\setminus Y)\subseteq\fcl(X)=A\); since
\(X\subseteq A\) and \(E\setminus X\subseteq B\), this forces \(A=E\).
Thus they can
be labelled and oriented as
\[
 X\subseteq Y,
 \qquad
 \fcl(X)=\fcl(Y)=A.
\tag{3.9}\label{eq:duplicate-sides}
\]
Every edge on the path between the two given edges
displays a set \(Z\) with \(X\subseteq Z\subseteq Y\).  If its separation is
non-sequential, monotonicity gives
\[
 A=\fcl(X)\subseteq\fcl(Z)\subseteq\fcl(Y)=A,
\]
so that edge displays the same class.

There is no flower vertex on this path.  Indeed, let \(w\) be such a flower
and let \(L,R\) be the two petals containing the two path directions.  A
displayed non-sequential separation which separates \(L\) from \(R\) has an
orientation \(Z\) between \(X\) and \(Y\), and hence has the same class.  If a
displayed side \(C\) does not separate \(L\) from \(R\), orient it so that
\(C\subseteq Y\setminus X\).  Then \(X\subseteq E\setminus C\), and therefore
\[
 A=\fcl(X)\subseteq\fcl(E\setminus C).
\]
Also \(E\setminus Y\subseteq E\setminus C\), while \(Y\subseteq A\).  Thus
\[
 E=A\cup(E\setminus Y)\subseteq\fcl(E\setminus C),
\]
so \(C\mid E\setminus C\) is sequential.  Consequently every
non-sequential class displayed at \(w\) would be the same class.  This
contradicts the fact that a flower vertex of \(\mathcal S\)-order at least
three displays at least two inequivalent non-sequential separations.

Hence the path consists only of bag vertices.  By the partial-tree edge
axiom, all its edges are non-sequential, and the preceding argument shows that
they all display the same class.  Contract all but one, merging bag labels.
This preserves every partial-tree axiom and every displayed equivalence
class.  Repeating for every strong class yields an equivalent maximal partial
tree in which strong classes and strong edges are in bijection.

Orient the unique edge of each strong class according to the anchor rule, and
let \(X_{\mathfrak s}\) be its displayed side.  Let
\(F_{\mathfrak s}\) be the selected canonical core.  The two laminar families
\[
 \{X_{\mathfrak s}\}
 \quad\text{and}\quad
 \{F_{\mathfrak s}\}
\tag{3.10}\label{eq:raw-core-families}
\]
have the same nesting and disjointness relations.  To see this, suppose the
selected orientations satisfy \(s\preceq t\).  Both corresponding raw edge
sides contain the nonempty core \(F_s\), so they meet.  Since tree-edge
separations do not cross, either \(X_s\subseteq X_t\),
\(X_t\subseteq X_s\), or \(X_s\cup X_t=E\).  In the second case, monotonicity
in both orientations, together with \(s\preceq t\), makes both members of the
two closure pairs equal, so the classes are equal.  In the third case,
\[
 \fcl(E\setminus X_t)\subseteq\fcl(X_s)\subseteq\fcl(X_t),
\]
which forces \(\fcl(X_t)=E\), contrary to non-sequentiality.  Thus, for
distinct classes,
\[
 s\preceq t\quad\Longrightarrow\quad X_s\subseteq X_t.
\tag{3.11}\label{eq:raw-order}
\]
Applying the same argument to \(s\preceq t^*\) gives
\(X_s\subseteq E\setminus X_t\), hence \(X_s\cap X_t=\varnothing\).  These
are exactly the relations satisfied by the corresponding cores.

Delete the unique strong edges and contract each remaining component to one
vertex.  If there is no strong class, this leaves one component, matching the
root-only inner skeleton of the set system consisting of \(E\) and the
singletons, so (i) is immediate.  Otherwise root the quotient at the component
outside the selected anchor side.  The resulting edge-labelled tree is the
Hasse tree of the selected raw sides.  By
\eqref{eq:raw-core-families}--\eqref{eq:raw-order}, that hierarchy is
canonically isomorphic to the inner-node tree of the selected core family.
Thus the quotient is canonically isomorphic to the inner-node skeleton of the
canonical-core tree.  Original-vertex leaves are attached separately according
to canonical-core membership.  This proves (i) without claiming that one
global partial tree literally displays every selected core.

Every non-sequential edge separation in a maximal partial tree is strong:
it is nested with every separation displayed by the tree, and
\cref{thm:CW} displays a representative of every non-sequential class.
Now suppose two flower nodes \(v,w\) survive in one component.  Let \(P\) be
the petal at \(v\) pointing toward \(w\).  Its edge is not a deleted
non-sequential edge, so its separation is sequential.  No-looseness gives
\(\fcl(P)\neq E\), hence \(\fcl(E\setminus P)=E\).  At \(w\), choose a
displayed non-sequential side \(B\) not using the petal toward \(v\).  Then
\(B\subseteq P\), so
\(E\setminus P\subseteq E\setminus B\).  Monotonicity would give
\(\fcl(E\setminus B)=E\), contradicting that \(B\) is non-sequential.
This proves (ii).

For (iii), a petal cut which is non-sequential is a strong tree-edge class
and hence becomes a boundary.  Every other petal cut is sequential; again
no-looseness rules out closure on the petal side, proving
\(\fcl(E\setminus P)=E\).  If a strong boundary core \(F\) lay strictly
inside such a petal, then
\(E\setminus P\subseteq E\setminus F\), contradicting
\(\fcl(E\setminus F)\neq E\).

For (iv), let \(X\mid E\setminus X\) be a canonical-block-saturated non-sequential
separation.  By \cref{thm:CW}, choose an equivalent displayed separation.
If its display is outside the component, let \(F\) be the first incident
boundary core toward that display.  Orient the remote class toward the
component boundary and let \(A\) be its canonical core.  The closure-pair
order gives \(A\subseteq F\).  Orient \(X\) as \(X_0\) containing \(A\);
canonical block-saturation then gives
\(F\subseteq X_0\).
Equivalence forces
\(\fcl(A)=\fcl(X_0)\): the opposite pairing, together with
\(A\subseteq X_0\), would put both \(X_0\) and its complement in
\(\fcl(X_0)\).  Therefore
\[
 \fcl(A)\subseteq\fcl(F)\subseteq\fcl(X_0)=\fcl(A),
\]
so the class is the boundary class of \(F\).  It cannot be displayed inside
the component, which has neither a flower nor an undeleted non-sequential
edge.  This proves (iv).
\end{proof}

The distinction in the last paragraph is important: Clark--Whittle display
equivalence does not say that all selected cores occur simultaneously as
literal edge sides.  We neither need nor use that assertion.  Instead we use
the partial tree only at class level and form the following algebraic local
model directly from the closure pairs.

\begin{remark}[Logical role of the auxiliary tree]\label{rem:auxiliary-tree}
The class-normal partial tree is an existence witness used only in
\cref{lem:canonical-component-blocks,lem:flower-torso,prop:bag-layout} to
prove a uniform bound on local linear rank-width.  The transduction in
\cref{cor:strong-tree} outputs only the canonical core-inclusion tree, and
the interpretation in \cref{lem:local-expansion} is defined directly from
that tree, the ambient adjacency relation, the coherent frame selectors,
and the child-state colours.  It never tests whether an auxiliary component
contains a flower and never selects a linear layout.  This is sufficient:
the theorem of Boja\'nczyk--Grohe--Pilipczuk supplies one CMSO formula valid
on the entire bounded-linear-clique-width class, so the existence of a
bounded-width layout is a semantic hypothesis and need not itself be
CMSO-definable.
\end{remark}

\begin{lemma}[Canonical component blocks]
\label{lem:canonical-component-blocks}
Assume \(|E|\geq9\), and fix a component node in the abstract quotient tree
of strong classes.  For
each incident strong class, orient it toward the branch away from the node and
let \(F_i\) be the core of that orientation.  Then the \(F_i\) are pairwise
disjoint ns-sides.  With
\[
 K=E\setminus\bigcup_iF_i,
\]
the class-normal partial-tree component determines whether this canonical
component is flower-free or has one flower.  In the latter case an incident
strong direction is represented by a flower petal \(P_i\) satisfying
\[
 F_i\subseteq P_i\subseteq\fcl(F_i),
\tag{3.12}\label{eq:petal-fringe}
\]
while every other petal is outward-sequential and contains no canonical
boundary core.
\end{lemma}

\begin{proof}
For two different incident directions, the corresponding oriented closure
pairs face each other in the compatibility order.  The facing case in the
proof of \cref{lem:anchor-laminar} shows that their cores are disjoint.  Each
core is an equivalent ns-side by the closure calculus.

The quotient of the class-normal partial tree and the canonical-core
inclusion tree are the same abstract tree by
\cref{prop:strong-normal}(i).  Thus the former supplies the bag/one-flower
label at the node without requiring literal equality of edge sides.  By
\cref{prop:strong-normal}(iii), a flower direction containing an incident
strong boundary cannot be outward-sequential; its petal cut therefore
represents that boundary class.  For the correctly oriented petal
representative \(P_i\), the core identity gives \(F_i\subseteq P_i\), while
equivalence gives \(\fcl(P_i)=\fcl(F_i)\).  This proves
\eqref{eq:petal-fringe}.  The remaining petal assertions are
\cref{prop:strong-normal}(iii).
\end{proof}

\section{Corrected torsos and bounded local linear width}

We now assume \(|E|\geq9\); smaller prime graphs are retained as bounded
base cases in \cref{prop:prime-definability}.  Fix one component of the
abstract strong-class tree and apply
\cref{lem:canonical-component-blocks}.  Denote its pairwise disjoint,
outward-oriented canonical boundary cores by
\(F_1,\ldots,F_m\), and put
\[
 K=E\setminus(F_1\cup\cdots\cup F_m).
\tag{4.1}\label{eq:component-partition}
\]
All boundary fringes remain in \(K\).

For each \(i\), the cut \(F_i\mid E\setminus F_i\) has rank two.  Choose an
ordered pair \(R_i=(r_i^1,r_i^2)\subseteq F_i\) whose two rows form a basis of
\(A_G[F_i,E\setminus F_i]\).  Let
\(\Lambda=\F_2^2\setminus\{0\}\).
By the definition of a core, \(E\setminus F_i\) is one of the two fully
closed members of its closure pair.

\begin{construction}[Corrected simultaneous torso]\label{con:torso}
Replace \(F_i\) by a terminal triple
\(T_i=\{t_i^a:a\in\Lambda\}\).  The vertex set of the torso \(H\) is
\[
 V(H)=K\ \dot\cup\ T_1\ \dot\cup\cdots\dot\cup\ T_m.
\]
Keep the graph induced by \(K\).  For \(x\in K\) and distinct \(i,j\), put
\begin{align}
 A_H(t_i^a,x)
   &=\sum_{p=1}^2 a_p A_G(r_i^p,x),
     \tag{4.2}\label{eq:core-terminal}\\
 A_H(t_i^a,t_j^b)
   &=\sum_{p,q=1}^2 a_pb_q A_G(r_i^p,r_j^q).
     \tag{4.3}\label{eq:terminal-terminal}
\end{align}
Make every \(T_i\) a triangle \(K_3\).
\end{construction}

When the same selected ordered basis is used for each occurrence of a
boundary core, the rooted local interface graph at a canonical-tree node is
obtained from its canonical component torso by deleting the parent terminal
triple and the fringe vertices outside the rooted territory.  Indeed,
replacing a canonical child core removes precisely that descendant territory,
while all fringe elements remain in \(K\); in the parent direction one uses
the core of the opposite orientation.  Terminal--terminal and terminal--\(K\)
adjacencies depend only on the ambient adjacency matrix and
the selected boundary bases.  The rooted local graph is therefore an induced
subgraph of the torso, and its linear rank-width cannot increase under these
vertex deletions.

The terminal--terminal term \eqref{eq:terminal-terminal} is indispensable:
different removed blocks can have nonzero adjacency to each other.

\begin{lemma}[Simultaneous block identity]\label{lem:block-identity}
For \(X\subseteq K\) and \(I\subseteq[m]\), let
\[
 Z=X\cup\bigcup_{i\in I}T_i,
 \qquad
 \widehat Z=X\cup\bigcup_{i\in I}F_i.
\]
Then
\[
 \rho_H(Z)=\rho_G(\widehat Z).
\tag{4.4}\label{eq:block-identity}
\]
\end{lemma}

\begin{proof}
In the cut matrix of \(\widehat Z\), replace the rows belonging to every
selected block \(F_i\) by the basis rows \(r_i^1,r_i^2\).  This preserves row
rank because all cut columns lie outside \(F_i\).  By symmetry, replace the
columns belonging to every unselected block \(F_j\) by the two basis columns
indexed by \(r_j^1,r_j^2\).  The resulting matrix has core entries from
\(G[K]\), core--block entries given by \eqref{eq:core-terminal}, and
block--block entries given by the bilinear matrix
\eqref{eq:terminal-terminal}.

In the cut matrix of \(Z\), the row or column indexed by
\(t_i^{(1,1)}\) is the sum of those indexed by
\(t_i^{(1,0)}\) and \(t_i^{(0,1)}\).  Delete all such dependent third rows and
columns.  What remains is exactly the reduced ambient cut matrix.  The
internal triangle edges never cross a block-saturated cut, proving
\eqref{eq:block-identity}.
\end{proof}

We first handle a canonical component containing a flower.

\begin{lemma}[Flower torso layout]\label{lem:flower-torso}
If the canonical component contains a flower, its torso has a linear layout
of cut-rank at most four in which every terminal triple is consecutive.
\end{lemma}

\begin{proof}
Consider first an incident strong petal \(P_i\).  By
\eqref{eq:petal-fringe}, its canonical core \(F_i\) satisfies
\(F_i\subseteq P_i\subseteq\fcl(F_i)\).  A partial 2-sequence orders
\(P_i\setminus F_i\), by \cref{lem:nested-equivalent}, so that every prefix
added to \(F_i\) remains
2-separating.  In the torso, order the three vertices of \(T_i\) first and
then that fringe sequence.  Prefixes inside \(T_i\) have cut-rank at most
two; after the whole triple is present, \cref{lem:block-identity} transfers
the partial sequence.  Hence the torso petal
\((P_i\setminus F_i)\cup T_i\) has a width-two ordering.

Every other petal \(P\) satisfies \(\fcl(E\setminus P)=E\) and contains no
boundary core.  Its absorption sequence adds only vertices of \(P\), and
every prefix keeps all boundary blocks whole.  It therefore transfers
through \cref{lem:block-identity}; by \cref{lem:reversal}, the corresponding
torso petal has a width-two ordering.

Replacing \(F_i\) by \(T_i\) inside its petal preserves the cut-rank of each
union of complete petals by \cref{lem:block-identity}; thus the modified
petals are still a flower.

Concatenate petal orderings in flower order.  A prefix inside a petal is the
union of a consecutive set of completed petals, of cut-rank at most two, and
a width-two prefix of the current petal.  Disjoint subadditivity gives rank at
most four.  Each terminal triple is the initial interval of its modified
petal ordering and hence is consecutive.
\end{proof}

It remains to treat flower-free components, for which we compress the
boundary blocks to atoms.  Suppose now that the component contains no flower.
Consider a partition
whose parts are some already-compressed boundary blocks, as single atoms,
and all remaining vertices as singleton atoms.  For a set \(\mathcal X\) of
atoms define
\[
 \mu(\mathcal X)=\rho_G\!\left(\bigcup_{A\in\mathcal X}A\right).
\tag{4.5}\label{eq:atom-connectivity}
\]

\begin{lemma}[One-at-a-time atom compression]\label{lem:atom-compression}
The boundary blocks can be compressed successively while preserving an atom
branch decomposition of width at most two, except that the process may stop
with a torso on at most twelve vertices.
\end{lemma}

\begin{proof}
Initially \(\mu=\rho_G\), so its branch-width is at most two.  At every stage
\(\mu^+=\mu+1\) is 3-connected in the convention of
Hall--Oxley--Semple--Whittle.  Indeed, a split-prime graph is connected, so a
nonempty proper atom union has ambient cut-rank at least one, while a
bipartition with at least two atoms on each side lifts to a partition with at
least two original vertices on each side and hence has ambient cut-rank at
least two by split-primality.  Thus \(\mu^+(\varnothing)=1\), every nonempty
proper atom union has value at least two, and every atom bipartition with at
least two atoms on each side has value at least three, exactly matching their
definition of 3-connectivity.

Let \(F_j\) be an uncompressed block.  It consists of at least four singleton
atoms by \cref{lem:small-sides}.  If its atom-complement also has at least
four atoms, then \(\mu^+\) has branch-width exactly three: the inherited
decomposition gives the upper bound, while every cubic decomposition on at
least eight atoms has an edge with at least two atoms on each side, whose
value is at least three by the preceding 3-connectivity.  Moreover,
\(\mu^+(F_j)=\rho_G(F_j)+1=3\).  Theorem~4.1 of
Hall--Oxley--Semple--Whittle applies to arbitrary connectivity functions and
states that a 3-separating set in a 3-connected function of branch-width
three can fail to be displayed only if it or its complement has two or three
atoms, all having singleton value two~\cite{HOSW02}.  Since both sides here
have at least four atoms, the contrapositive gives a width-three branch
decomposition displaying \(F_j\).  Contract the
displayed \(F_j\)-branch to one atom leaf.  Every surviving displayed cut is
an old union-of-atoms cut, so the unshifted width remains at most two.

If the complement of \(F_j\) has at most three atoms, no other uncompressed
boundary block remains, since such a block alone has at least four vertices.
Write the complement as \(r\) previously compressed block-atoms and \(s\)
interior singleton atoms.  Then \(r+s\leq3\), and after replacing all block
atoms, including \(F_j\), by triples, the torso has
\[
 3(r+1)+s\leq12
\]
vertices.  This is the exceptional bounded case.
\end{proof}

Assume first that every block was compressed, and let \(D\) be the resulting
width-two branch decomposition of the final atom system.  A cut of \(D\) is
\emph{admissible}: every \(F_i\) lies wholly on one side.

\begin{lemma}[Location of an admissible ns cut]\label{lem:admissible-location}
If the ambient lift of a displayed atom cut is non-sequential, then its class
is an incident boundary class.
\end{lemma}

\begin{proof}
This is \cref{prop:strong-normal}(iv).  For completeness, an equivalent
Clark--Whittle display outside the component lies beyond a first boundary
class.  Let \(R\) be its remote displayed side, let \(P_i\) be the
tree-displayed side of that boundary, and put
\(A=E\setminus\fcl(E\setminus R)\).  Since
\(E\setminus P_i\subseteq E\setminus R\) and the core of \(P_i\) is \(F_i\),
monotonicity gives
\[
 E\setminus F_i=\fcl(E\setminus P_i)
 \subseteq\fcl(E\setminus R),
\]
hence \(A\subseteq F_i\).  Orient the admissible side \(X\) to contain
\(A\); block-saturation then forces \(F_i\subseteq X\).  Equivalence cannot pair
\(\fcl(A)\) with \(\fcl(E\setminus X)\), since monotonicity would then force
\(\fcl(X)=E\).  Hence
\(\fcl(A)\subseteq\fcl(F_i)\subseteq\fcl(X)=\fcl(A)\), identifying the
boundary class.  There is no flower or undeleted non-sequential edge inside
the component.
\end{proof}

The next small tree argument is needed because a representative equivalent
to a boundary cut could a priori contain several block atoms.

\begin{lemma}[One-block boundary representative]\label{lem:one-block}
Let \(Q\mid\mathcal A\setminus Q\) be a cut displayed by \(D\), and suppose
its ambient lift \(X\) is equivalent to the boundary class of \(F_i\).  Orient
it with \(F_i\subseteq X\).  Then no other boundary block lies in \(X\).
\end{lemma}

\begin{proof}
The orientation gives \(\fcl(X)=\fcl(F_i)\); the opposite pairing would make
that closure contain both \(F_i\) and its complement.  Thus
\(X\subseteq\fcl(F_i)\).  For \(j\neq i\), disjointness gives
\(F_i\subseteq E\setminus F_j\), and \(E\setminus F_j\) is fully closed
because \(F_j\) is a canonical core.  Hence
\[
 \fcl(F_i)\subseteq E\setminus F_j.
\]
Consequently \(X\) cannot contain \(F_j\).
\end{proof}

\begin{lemma}[Displayed atom cuts are orientable]\label{lem:atom-orientable}
Every edge cut of \(D\) has a side whose full closure in the atom system is
all atoms.
\end{lemma}

\begin{proof}
Let \(X\) be the ambient lift.  If \(\fcl_G(X)=E\), then \(X\) contains every
boundary block.  Otherwise, for a missing \(F_i\), one would have
\(X\subseteq E\setminus F_i\) and hence
\(E=\fcl(X)\subseteq\fcl(E\setminus F_i)\neq E\).  An absorption sequence
from \(X\) consequently adds only interior singleton atoms and transfers
step-for-step to \(\mu\).  The same argument applies when the complementary
ambient side has full closure.

It remains to consider an ambient non-sequential cut.  By
\cref{lem:admissible-location}, it is a boundary class, say that of \(F_i\).
By \cref{lem:one-block}, orient it as
\[
 X=F_i\cup Y,\qquad Y\subseteq K,
\]
with every other boundary block in \(R=E\setminus X\).  Since \(F_i\) is a
canonical core, \(H_i=E\setminus F_i\) is fully closed.  Class equivalence
gives \(\fcl(R)=H_i\), so an absorption sequence from \(R\) to \(H_i\) adds
only the interior elements of \(Y\).  It transfers to the atom system while
all other block atoms remain fixed.  Finally add the single atom \(F_i\); the
connectivity changes from two to zero.  Thus the atom complement of \(Q\)
has full closure.

If a displayed cut has an underlying side of one original vertex, its
complement absorbs that singleton immediately.  These cases exhaust the
edges of \(D\).
\end{proof}

\begin{proposition}[Port-contiguous bag layout]\label{prop:bag-layout}
Every flower-free normalized torso has a linear layout of cut-rank at most
six in which every terminal triple is consecutive.
\end{proposition}

\begin{proof}
In the nonexceptional case, apply \cref{lem:sink} to \(D\), using
\cref{lem:atom-orientable}.  This yields an atom ordering of width at most
four.  Replace every block atom \(F_i\) in place by the three members of
\(T_i\).  At a boundary between completed groups,
\cref{lem:block-identity} preserves the atom cut-rank, hence gives rank at
most four.  A prefix one or two vertices into the current triple has
cut-rank at most \(4+1\) or \(4+2\), respectively, by disjoint
subadditivity.  The resulting width is at most six.

In the exceptional case of \cref{lem:atom-compression}, the torso has at
most twelve vertices.  Put every triple consecutively in an arbitrary
ordering.  Every cut has rank at most the smaller side, hence at most six.
\end{proof}

\begin{corollary}[Uniform local bound]\label{cor:local-width}
Every rooted local torso associated with the strong-core tree is an induced
subgraph of a canonical component torso and has a port-contiguous layout of linear
rank-width at most six.
\end{corollary}

\begin{proof}
Use \cref{lem:flower-torso} for a flower component and
\cref{prop:bag-layout} otherwise, then delete the parent interface and any
outside fringe vertices not belonging to the rooted territory.
\end{proof}

\section{Coherent frames and finite-state evaluation}

Let \(\mathcal D\) be the laminar tree produced in
\cref{cor:strong-tree}.  Its leaves are the vertices of \(G\).  If \(u\) is
an inner node, write \(S_u\) for the set of descendant leaves, let
\(\operatorname{Ch}(u)\) be its inner-node children, and let \(I_u\) be its
singleton children.  Thus
\[
 S_u=I_u\ \dot\cup\!\bigcup_{v\in\operatorname{Ch}(u)} S_v.
\tag{5.1}\label{eq:territory-partition}
\]
Every nonroot \(S_u\) is the canonical core of a strong exact
rank-two cut.

We first use two selector colours to make the ordered boundary frames
coherent.  The main logical danger is to choose an unrelated basis at the two
occurrences of one tree edge.  Orbit saturation does not fix this: a
nontrivial element of \(\operatorname{GL}(2,2)\) can change the graph obtained
after substitution.  The following lemma supplies literal shared names.

\begin{lemma}[Coherent frame selectors]\label{lem:coherent-frames}
Two unary colours \(C_1,C_2\) on the nodes of \(\mathcal D\) suffice to
select, for every nonroot inner node \(u\), an ordered row basis
\[
 R_u=(r_1(u),r_2(u))\subseteq S_u
\]
of \(A_G[S_u,E\setminus S_u]\), with the coherence property
\[
 r_j(u)\in S_v\text{ for an inner child }v\in\operatorname{Ch}(u)
 \quad\Longrightarrow\quad r_j(v)=r_j(u).
\tag{5.2}\label{eq:frame-coherence}
\]
The validity of the selection is first-order over the graph equipped with
the laminar-tree ancestor relation.
\end{lemma}

\begin{proof}
Require that, for each nonroot inner node \(u\) and each \(j\in\{1,2\}\),
exactly one child of \(u\) has colour \(C_j\).  A child may have both colours.
Define \(\operatorname{Rep}_j(u,x)\) to mean that \(x\) is a singleton
descendant of \(u\) and every node on the path strictly below \(u\) and down
to \(x\) has colour \(C_j\).  The unique-child requirement makes this
singleton unique; call it \(r_j(u)\).  Since the output vocabulary contains
the ancestor relation, the path condition is first-order.

Filter the colouring by requiring that the rows of \(r_1(u),r_2(u)\) on
\(E\setminus S_u\) are nonzero and distinct, and that every row indexed by
\(S_u\) is one of
\[
 0,\quad r_1(u),\quad r_2(u),\quad r_1(u)+r_2(u)
\]
on those columns.  These are finite first-order row-equality tests.  They say
exactly that \(R_u\) is an ordered basis.

It remains to prove that a valid colouring exists.  Work top-down, starting
independently at the nonroot children of the root.  Suppose a basis has been
chosen at \(u\), and consider an inner child \(v\).  If one selected element
of \(R_u\) lies in \(S_v\), its row remains nonzero when the column set is
enlarged from \(E\setminus S_u\) to \(E\setminus S_v\).  If both lie in
\(S_v\), their independence is preserved for the same reason.  Retain every
inherited element in its slot and, when only one is inherited, extend it to a
basis of the rank-two child cut; when none is inherited, choose any ordered
basis.  Colour the immediate child containing each chosen representative.
Continuing recursively produces the required monochromatic selector paths,
and \eqref{eq:frame-coherence} is automatic.
\end{proof}

For \(x\in S_u\), let \(\ell_u(x)\in\F_2^2\) be the unique coefficient vector
such that
\[
 A_G[x,E\setminus S_u]
 =\ell_u(x)_1A_G[r_1(u),E\setminus S_u]
 +\ell_u(x)_2A_G[r_2(u),E\setminus S_u].
\tag{5.3}\label{eq:port-label}
\]
Each of the four predicates \(\ell_u(x)=a\) is first-order over the decorated
tree.  We regard \(G[S_u]\), coloured by \(\ell_u\), as a graph with port
alphabet \(\Gamma=\F_2^2\).

Put \(e_1=(1,0)\) and \(e_2=(0,1)\).

\begin{definition}[Valid rank-two port graph]\label{def:valid-port}
A canonical \(\Gamma\)-labelled graph \(H\) is a \emph{valid rank-two port
graph} if there are a graph \(H^+\) containing \(H\) as an induced subgraph,
a set
\[
 O=V(H^+)\setminus V(H),
\]
and vertices \(b_1,b_2\in V(H)\) such that the rows of \(b_1,b_2\) form an
ordered basis of \(A_{H^+}[V(H),O]\), and, whenever \(x\in V(H)\) has label
\(a=(a_1,a_2)\), one has
\[
 A_{H^+}[x,O]
 =a_1A_{H^+}[b_1,O]+a_2A_{H^+}[b_2,O].
\]
\end{definition}

A canonical \(\Gamma\)-labelled graph is valid if and only if both label
classes \(e_1\) and \(e_2\) are nonempty.  For the forward implication,
independence forces the labels of \(b_1,b_2\) to be \(e_1,e_2\),
respectively.  Conversely, choose vertices \(b_j\) of label \(e_j\), add two
new vertices \(o_1,o_2\), and join a vertex of label \(a\) to \(o_j\) exactly
when \(a_j=1\).  The resulting outside cut has rank two and realizes the
prescribed port labels.

Coherent frames now permit a single finite port congruence for the local
computation.  Fix a VR-recognizable property \(\mathcal P\).  Choose once and for all a
finite label alphabet \(\Sigma\) containing disjoint namespaces for
\begin{enumerate}[label=\textup{(\alph*)}]
  \item a canonical copy \(\Gamma\) of \(\F_2^2\);
  \item a temporary child copy \(\Gamma_{\rm ch}\); and
  \item one dummy root output label \(\bot\), and the temporary outside
        labels \((c,s)\in(\Gamma\cup\{\bot\})\times\F_2^2\).
\end{enumerate}
Fix a locally finite congruence \(\equiv\) of the full VR algebra
recognizing \(\mathcal P\), and let \(\equiv_\Sigma\) be its restriction to
the \(\Sigma\)-labelled sort.  This restriction has finite index and
saturates the inverse image of \(\mathcal P\) under forgetting labels.
Restrict its classes to graphs labelled in the canonical namespace
\(\Gamma\), and call the resulting finite state set \(\widehat Q\).
Each state class is a globally VR-recognizable coloured language: it is
recognized by the same locally finite quotient after the fixed inclusion
of \(\Gamma\) into \(\Sigma\).  Compatibility with polynomial contexts
allows fixed outside operands of arbitrary clique-width.

Let \(Q\subseteq\widehat Q\) be the states realized by valid rank-two port
graphs.  Only these states are needed.  For every \(q\in Q\), fix a valid
four-labelled representative \(J_q\).  By \cref{def:valid-port}, the labels
\(e_1\) and \(e_2\) are nonempty in every \(J_q\).  Put
\[
 M=\max_{q\in Q}|V(J_q)|.
\tag{5.4}\label{eq:M}
\]
The representatives and all their internal adjacencies are fixed finite
data depending on \(\mathcal P\), not on the input graph.

This fixed congruence determines a uniform local expansion at every tree
node.  Suppose each inner child \(v\) of \(u\) has been assigned a state \(q_v\).
Define a labelled graph \(L_u\) as follows.  It is canonically
\(\Gamma\)-labelled when \(u\) is nonroot and dummy-labelled when \(u\) is
the root.  Its vertices are the owned vertices \(I_u\), together with one
private copy of \(J_{q_v}\) for every inner child \(v\).  Edges inside
\(I_u\) are inherited from \(G\), and each copy of \(J_{q_v}\) has its fixed
internal edges.

If \(z\in J_{q_v}\) has label \(a\in\F_2^2\) and \(x\in I_u\), put
\[
 A_{L_u}(z,x)=
 a_1A_G(r_1(v),x)+a_2A_G(r_2(v),x).
\tag{5.5}\label{eq:gadget-core}
\]
If \(z\in J_{q_v}\) and \(z'\in J_{q_w}\), with \(v\neq w\) and labels
\(a,b\), put
\[
 A_{L_u}(z,z')=
 a^{\mathsf T}
 \bigl[A_G(r_p(v),r_q(w))\bigr]_{p,q\in[2]}
 b.
\tag{5.6}\label{eq:gadget-gadget}
\]
For nonroot \(u\), give \(x\in I_u\) the output label \(\ell_u(x)\).  If
\(c_p=\ell_u(r_p(v))\), give a vertex of the \(v\)-gadget with child label
\(a\) the parent label
\[
 a_1c_1+a_2c_2.
\tag{5.7}\label{eq:label-transport}
\]
At the root, put every vertex into the single dummy namespace; the root
language ignores that label.

Every nonroot valid expansion is again a valid port graph.  Indeed, for
\(j\in\{1,2\}\), either \(r_j(u)\in I_u\), in which case it has output label
\(e_j\), or \(r_j(u)\in S_v\) for an inner child
\(v\in\operatorname{Ch}(u)\).  Coherence gives
\(r_j(v)=r_j(u)\); the representative \(J_{q_v}\) contains an \(e_j\)-labelled
vertex, and \eqref{eq:label-transport} sends that vertex to output label
\(e_j\).  Hence the state of \(L_u\) belongs to the chosen valid part of \(Q\).

\begin{lemma}[Uniform interpretation and width]\label{lem:local-expansion}
For a node parameter \(u\), unary state colours on its inner children, and
an assignment \(q_v\in Q\) to every such child, the graph \(L_u\) is
produced by one fixed finite-copy \(\CMSO\)-interpretation.  This includes
the dummy-labelled expansion at the root.  Every resulting expansion
satisfies
\[
 \lrw(L_u)\leq6+M,
\tag{5.8}\label{eq:expansion-width}
\]
and, with its fixed finite output-label alphabet, has linear clique-width
bounded by a constant \(b_{\mathcal P}\).
\end{lemma}

\begin{proof}
Keep the singleton children of \(u\), and make at most \(M\) fixed copies of
every inner child node, retaining the number and fixed vertex labels
prescribed by its state \(q_v\).  Equations
\eqref{eq:gadget-core}--\eqref{eq:label-transport} are finite disjunctions of
first-order formulas using \(\operatorname{Rep}_j\), and internal gadget
edges are hard-coded.  This is a uniform finite-copy interpretation; it does
not copy a ground vertex once for every ancestor.

Start from the port-contiguous torso layout of
\cref{cor:local-width}.  Replace the interval occupied by each terminal
triple by the corresponding \(J_{q_v}\).  At a cut between whole gadgets,
group the cut matrix by the replaced blocks.  The external row of a gadget
vertex labelled \(a\) is the \(a\)-linear combination of the two old terminal
rows.  Conversely, those two rows occur because every representative contains
the labels \(e_1,e_2\).  Thus the grouped cut has exactly the old rank and in
particular rank at most six.  At a cut inside one gadget, moving the at most
\(M\) vertices of that gadget across the nearest grouped cut changes rank by
at most \(M\), by subadditivity.  This proves
\eqref{eq:expansion-width}.  The argument uses only the underlying adjacency
relation and therefore applies unchanged to the dummy-labelled root
expansion.  The standard inequality
\(\lcw(H)\leq 2^{\lrw(H)}+1\)~\cite{Oum17}, followed by refinement with the
fixed finite output-label alphabet, supplies \(b_{\mathcal P}\).
\end{proof}

\begin{lemma}[Port substitution]\label{lem:port-substitution}
If a child graph is replaced by the representative of its state, the state
of the whole framed parent graph does not change.
\end{lemma}

\begin{proof}
Consider one concrete replacement of a child \(v\).  Rename the canonical
labels in the hole from \(\Gamma\) to \(\Gamma_{\rm ch}\).  Regard the already
fixed outside graph as a second operand whose vertex \(x\) is labelled in
advance by
\[
 (c(x),s(x)),\qquad
 s(x)=\bigl(A(r_1(v),x),A(r_2(v),x)\bigr),
\]
where \(c(x)\) is its persistent output label.  This does not assert that a VR
relabeling can split a pre-existing colour class: for this particular context
the outside operand is simply supplied with these finite profile labels.  Join
\(a\in\Gamma_{\rm ch}\) to \((c,s)\) exactly when \(a\cdot s=1\), then relabel
\(a\) to its transported parent label from
\eqref{eq:label-transport} (or to \(\bot\) at the root) and \((c,s)\) back to
\(c\).  The result is a unary VR polynomial over the
\(\Sigma\)-labelled sort, with the outside graph as a fixed operand and
the child as its hole.

Consequently \(\equiv_\Sigma\) permits replacement of the child by
\(J_{q_v}\).  Replace the children one at a time.  Before each replacement,
recompute the two-bit profiles in the current concrete outside graph; for a
previously replaced gadget these are precisely the profiles prescribed by
\eqref{eq:gadget-gadget}.  Thus every step is a legitimate congruence context,
and after the last step the graph is exactly \(L_u\).
\end{proof}

We finally express the resulting local-state predicates in \(\CMSO\).  For
\(q\in Q\), the set of canonical \(\Gamma\)-labelled graphs of state \(q\)
is globally VR-recognizable by the chosen quotient.  The proof of
Theorem~3.5 of Boja\'nczyk--Grohe--Pilipczuk gives the following labelled
relative form, using their decomposer Theorem~3.3, acceptance Claim~3.6,
and Transfer and Color Enforcement Lemmas~5.4--5.5~\cite{BGP21}:
for every fixed finite label alphabet, every globally VR-recognizable language
\(\mathcal K\), and every fixed bound \(b\), there is a \(\CMSO\) sentence
\(\psi\) such that
\[
 H\models\psi\quad\Longleftrightarrow\quad H\in\mathcal K
 \qquad\text{whenever }\lcw(H)\leq b.
\]
Apply this statement with \(b=b_{\mathcal P}\), obtaining a sentence
\(\psi_q\) for every state \(q\).  Obtain similarly a sentence
\(\psi_{\mathrm{acc}}\) for the recognizable dummy-labelled root language
whose underlying unlabelled graph lies in \(\mathcal P\).

Backward-translate \(\psi_q\) through the interpretation of
\cref{lem:local-expansion}.  This gives a formula
\(\operatorname{Transition}_q(u)\) over the decorated ambient graph, with
the child-state colours as parameters.  It says that the local expansion at
\(u\) has state \(q\).  At the root use the backward translation of
\(\psi_{\mathrm{acc}}\).

Guess unary predicates \((P_q)_{q\in Q}\) partitioning the nonroot inner
nodes, and require
\[
 \forall u\in\operatorname{Inn}(\mathcal D)\setminus\{\mathrm{root}\}\ 
 \bigwedge_{q\in Q}
 \bigl(P_q(u)\rightarrow\operatorname{Transition}_q(u)\bigr),
\tag{5.9}\label{eq:run}
\]
together with the accepting root formula, where
\(\operatorname{Inn}(\mathcal D)\) denotes the inner nodes.

\begin{proposition}[Prime recognizable implies definable]
\label{prop:prime-definability}
For every VR-recognizable property \(\mathcal P\), there is a \(\CMSO\)
sentence which defines \(\mathcal P\) on split-prime graphs of rank-width at
most two.
\end{proposition}

\begin{proof}
Suppose first that the state colouring satisfies \eqref{eq:run}.  Induct
upwards in the laminar tree.  By the induction hypothesis, each child colour
is the state of its actual four-labelled territory.  Replace the child
territories one at a time by their representatives.  By
\cref{lem:port-substitution}, the state does not change, and the resulting
graph is exactly \(L_u\).  The transition formula therefore assigns to \(u\)
the state of its actual framed territory.  At the root, the accepting formula
holds exactly when \(G\in\mathcal P\).

Conversely, choose the actual state of every framed territory.  The coherent
frames exist by \cref{lem:coherent-frames}, every transition is correct by
\cref{lem:port-substitution}, and the root is accepting whenever
\(G\in\mathcal P\).  All guesses made by the construction are unary
colourings or the single anchor set, so composition and backward translation
produce one \(\CMSO\) sentence.  Prime graphs of bounded exceptional order
are handled by a finite disjunction.
\end{proof}

\section{From prime graphs to all graphs}

The proof of \cref{prop:prime-definability} is unchanged after adding any
fixed finite collection of unary vertex labels: cut-rank and closure ignore
the labels, the port alphabet is replaced by a fixed finite product, and the
same finite-state argument applies.  We use this labelled version for marker
vertices in a split decomposition.

Campbell--Guillon--Kant\'e--Kim--K\"ohler give a nondeterministic
\(\CMSO\)-transduction of the enriched canonical split decomposition
\cite[Theorem~5.9]{CampbellEtAl26}.  Their theorem is stated for strongly
connected directed graphs.  A connected undirected graph is interpreted as
the symmetric directed graph obtained by replacing every edge by its two
orientations; it is strongly connected, and the directed split relation then
specializes to the usual undirected split relation.  This first-order change
of vocabulary, followed by forgetting orientations in the factors, therefore
gives exactly the enriched canonical split decomposition used here.  Its
factor tree has prime factors and degenerate clique or star factors, and the
original graph is recovered by the fixed split-composition operation.

\begin{lemma}[Rank-width of split factors]\label{lem:split-factor-width}
Every prime factor in the canonical split decomposition of a graph \(G\) has
rank-width at most \(\rw(G)\).  The clique and star factors have rank-width at
most one.
\end{lemma}

\begin{proof}
For a prime factor \(B\), choose behind each marker one original vertex in
the nonempty active class of the corresponding split side.  Between two
marker branches, split composition makes the two active classes either
complete or anticomplete according to whether the markers are adjacent in
\(B\).  The subgraph induced by the chosen representatives is consequently
isomorphic to \(B\).  Induced-subgraph monotonicity gives
\(\rw(B)\leq\rw(G)\), consistently with the characterization by prime
induced subgraphs in Theorem~4.3 of
Hlin\v{e}n\'y--Oum--Seese--Gottlob~\cite{HOSG08}.
Cliques and stars are distance-hereditary and hence have
rank-width at most one~\cite{Oum05}.
\end{proof}

Campbell et al. prove a general decomposition-level lifting theorem when a
CMSO transduction of bounded-width rank decompositions is already available
on the prime factors~\cite[Theorem~6.1]{CampbellEtAl26}.  Our prime argument
instead supplies, separately for each recognizable property, CMSO formulas
for its finite port states; it does not transduce a rank decomposition of the
prime graph.  The following property-level lift is therefore stated and
proved explicitly.

\begin{proposition}[Split lift]\label{prop:split-lift}
Suppose that, for every fixed finite vertex-label alphabet, every
VR-recognizable language is \(\CMSO\)-definable on the class of labelled
split-prime graphs of rank-width at most two.  Then every VR-recognizable
property is \(\CMSO\)-definable on all graphs of rank-width at most two.
\end{proposition}

\begin{proof}
We first treat connected graphs.  Fix the target language
\(\mathcal P\), and choose one locally finite congruence of the full
many-sorted VR algebra recognizing it.  Use its restrictions to the
finitely many working sorts below.  Every state class in these sorts is
globally VR-recognizable by the same quotient.  Put
\(\Delta=\{0,1\}\), and let \(Q_0\) be the
finite set of congruence states realized by unlabelled graphs.  Let
\(q_\bullet\) be the state of the one-vertex graph carrying label \(1\).
Define \(Q_1\) to consist of \(q_\bullet\) together with all congruence states
of \(\Delta\)-labelled graphs that occur as nontrivial split sides, where
label \(1\) denotes the nonempty active class of the side.  For every
\(q\in Q_1\setminus\{q_\bullet\}\), fix a \(\Delta\)-labelled representative
\(J_q\) having a \(1\)-labelled vertex, and take \(J_{q_\bullet}\) to be the
one-vertex graph labelled \(1\).  Finally, let \(A\subseteq Q_0\) be the set
of accepting states.

We use the following rank-one port operation.  If \(H\) is a
\(\Delta\)-labelled graph and \(D\) is any disjoint labelled graph with a
specified vertex set \(Z\), take their disjoint union, join every
\(1\)-labelled vertex of \(H\) to every vertex of \(Z\), and then apply fixed
relabelings.  This is a unary VR polynomial in \(H\), with the labelled
outside graph as a fixed operand.  Consequently
\(H\) may be replaced by \(J_q\) in this operation whenever \(H\) has state
\(q\).  The second operand \(D\) is arbitrary; in particular it need not have
bounded clique-width.

The finitely many connected inputs for which the canonical split
decomposition has no non-leaf factor are handled directly by a finite
disjunction.  We henceforth assume that a non-leaf factor exists.

In the enriched split decomposition, write \(x\sim_c y\) when \(x\) and
\(y\) lie in the same connected component of the graph formed by the
c-edges.  The relation \(\sim_c\) is \(\MSO\)-definable, since c-edge
connectivity is expressible by quantifying over vertex sets with no c-edge
crossing.  Its non-leaf equivalence classes are precisely the factor graphs,
while the remaining singleton classes correspond to the original-vertex
leaves.  Contracting the \(\sim_c\)-classes and retaining the t-edges gives
the factor tree of the canonical split decomposition.

For a connected graph having a non-leaf factor, guess one vertex
\(r_{\mathrm{root}}\) in such a factor and root the quotient tree at its
\(\sim_c\)-class.  For every nonroot inner factor class \(t\), let \(B_t\)
be its c-edge graph.  The unique t-edge from \(t\) toward the root has one
endpoint in \(B_t\); denote this endpoint by \(p_t\), the parent marker of
\(B_t\).  For each \(x\in V(B_t)\setminus\{p_t\}\), the unique t-edge
incident with \(x\) leads either to a child factor or to an original-vertex
leaf.  All these relations are \(\MSO\)-definable from \(\sim_c\), the
t-edge relation, and the chosen root class.  Give \(x\) the state \(q_x\) of
the corresponding branch, using \(q_\bullet\) when the t-edge leads to an
original-vertex leaf, and put
\[
 \epsilon_t(x)=A_{B_t}(x,p_t).
\]
Define the one-port expansion \(\mathsf E_t(B_t)\) as follows.  Delete
\(p_t\), replace every remaining marker \(x\) by a private copy of
\(J_{q_x}\), and, for distinct markers \(x,y\), join the \(1\)-parts of their
copies completely if and only if \(xy\in E(B_t)\).  A vertex of the
\(x\)-copy whose label in \(J_{q_x}\) is \(a\in\Delta\) receives output label
\[
 a\epsilon_t(x).
\tag{6.1}\label{eq:split-port-transport}
\]
At the root use the same expansion without a parent marker and forget the
output labels.

For \(q\in Q_1\), let \(\mathcal K_q\) be the language of finitely labelled
pointed factors for which \(\mathsf E_t(B_t)\) has state \(q\).  For
\(q\in Q_0\), define the root language \(\mathcal K_q^0\) analogously.  These
languages are VR-recognizable.  Here are the algebraic details.  Regard the
child-state predicates and the predicates \(\epsilon_t\) as immutable unary
colours, separate from the auxiliary labels of a VR term.  In a term for the
coloured factor, replace creation of a vertex of state \(q_x\) by the fixed
term for \(J_{q_x}\).  Its vertices carry product labels \((i,c,a)\),
where \(i\) is the current construction label, \(c\) records the immutable
child-state, parent-marker and \(\epsilon_t\) colours, and \(a\) is the
representative's activity bit.  A join between construction labels \(i,j\)
becomes all joins between \((i,c,1)\) and \((j,d,1)\), over the fixed
colour combinations \(c,d\).  Relabellings act only on \(i\).
The specially coloured parent marker is sent to the empty graph, and
\eqref{eq:split-port-transport} is the final relabelling
\((i,c,a)\mapsto a\epsilon(c)\), where \(\epsilon(c)\) is the bit
encoded by \(c\), with labels forgotten at the root.
This term translation evaluates exactly to \(\mathsf E_t(B_t)\).

To verify global recognizability, including relabellings of the input
colours, let \(C\) be the finite alphabet of decoration types and \(h\)
the homomorphism onto the chosen locally finite quotient.  For every
finite input alphabet \(L\) and every map \(\tau:L\to C\), let
\(E_\tau(H)\) be the expansion decorating a vertex of label \(i\) by
\(\tau(i)\), retaining the output product labels
\((i,\tau(i),a)\).  Record \(h(E_\tau(H))\) for every such map \(\tau\).
This finite tuple has only finitely many values on the \(L\)-sort.
Disjoint union and joins act on its entries by the translated operations
above.  Under a relabelling \(r:L\to L'\), the entry for
\(\tau:L'\to C\) is obtained from the entry for \(\tau\circ r\) by
relabelling the first output coordinate.  Equality of the tuples is
therefore a locally finite congruence of the full VR algebra.  In the
decoration sort \(C\), the entry \(\tau=\mathrm{id}_C\), followed by
\((i,c,a)\mapsto a\epsilon(c)\), determines the state of the expansion.
Hence every \(\mathcal K_q\) is globally VR-recognizable.  Forgetting the
output labels gives the same conclusion for the root languages.
Intersect with the recognizable validity condition that the state
colours partition \(V(B_t)\setminus\{p_t\}\), the parent marker is unique,
and \(\epsilon_t(x)\leftrightarrow xp_t\in E(B_t)\).  At the root there is
no parent marker.

If \(B_t\) is prime, its underlying graph has rank-width at most two by
\cref{lem:split-factor-width}.  Hence
the labelled version of \cref{prop:prime-definability} supplies a \(\CMSO\)
formula for each \(\mathcal K_q\).  If \(B_t\) is a clique or a star, its
linear clique-width remains bounded after adding the fixed finite collection
of unary colours, so the corresponding formulas follow from
Boja\'nczyk--Grohe--Pilipczuk.  We therefore have, for every possible state,
one local transition formula valid for every factor type.  These formulas may
be relativized to the c-edge components of the enriched split decomposition.

Transduce the enriched split decomposition and form the rooted quotient tree
described above.  Guess unary predicates \((P_q)_{q\in Q_1}\), require each
\(P_q\) to be supported on the parent markers of nonroot inner factors, and
require that exactly one \(P_q\) holds at every such parent marker.  The
interpretation is that \(P_q(p_t)\) records the state of the entire territory
rooted at \(t\).

At a marker \(x\in V(B_t)\setminus\{p_t\}\), let \(y\) be its t-neighbour.
If \(y\) is the parent marker of a child factor, assign to \(x\) the unique
state \(q\) for which \(P_q(y)\) holds.  If \(y\) belongs to an
original-vertex leaf class, assign \(q_\bullet\) to \(x\).  Distinguish
\(p_t\), define \(\epsilon_t(x)\) by adjacency to \(p_t\), and, for every
\(q\in Q_1\), require that \(P_q(p_t)\) implies the formula defining
\(\mathcal K_q\), relativized to the \(\sim_c\)-class of \(p_t\).

At the root factor, use the root formulas \(\mathcal K_q^0\), each
relativized to the root \(\sim_c\)-class and supplied with the child-state
predicates derived from the t-neighbours of its markers as above.  For the
connected acceptance construction require
\[
 \bigvee_{q\in A}\mathcal K_q^0.
\]
For the component-evaluation construction, guess \(q\in Q_0\), require
\(\mathcal K_q^0\), and place a predicate \(R_q\) on the uniquely chosen
vertex \(r_{\mathrm{root}}\).  Thus every connected component receives
exactly one root representative carrying its unlabelled congruence state.
All these are fixed \(\CMSO\) conditions on the transduced structure and
hence backward-translate to the original graph.

Correctness is by induction up the rooted factor tree.  The actual territory
below a nonroot factor is \(\Delta\)-labelled by whether a vertex is active
across its parent split.  By induction, each child territory has the state
written on its parent marker, which is the t-neighbour of the corresponding
marker of \(B_t\).  The definition of split decomposition says
that two child territories are joined exactly by the complete bipartite graph
between their active classes when the corresponding markers are adjacent in
\(B_t\).  It also says that activity toward the parent is transported exactly
by \eqref{eq:split-port-transport}.  Replacing the child territories one at a
time by their representatives therefore preserves the state by the rank-one
port operation above, and the resulting graph is precisely
\(\mathsf E_t(B_t)\).  The local formula consequently assigns the actual
territory state.  At the root, the unlabelled expansion has the state of the
whole connected graph.  Conversely, assigning every territory its actual
state satisfies all local formulas.  Thus an accepting root state is
equivalent to membership in \(\mathcal P\).

Finally let \(G\) be disconnected.  The relation of belonging to the same
connected component of \(G\) is \(\MSO\)-definable.  Apply the connected
component-evaluation transduction in parallel to all nontrivial connected
components, using the Parallel Application Lemma
\cite[Lemma~6.2]{CampbellEtAl26}.  If a connected component is an isolated
vertex, use that vertex itself as its unique root representative and give it
the state \(q_\circ\in Q_0\) of the unlabelled one-vertex graph.  Any other
bounded connected case omitted by the factor-tree construction is handled
locally by a finite disjunction, after which one of its vertices is chosen as
its unique root representative.  We therefore obtain unary predicates
\((R_q)_{q\in Q_0}\) whose union contains exactly one vertex from each
connected component, with \(R_q\) containing precisely the roots of
components of state \(q\).

Disjoint union induces a well-defined finite commutative monoid
\((Q_0,\oplus)\), where
\[
 [H]\oplus[K]=[H\mathbin{\dot\cup}K].
\]
The state of \(G\) is the product of the states carried by its component
roots.  For each \(q\in Q_0\), choose integers \(N_q\geq0\) and \(p_q\geq1\)
such that
\[
 q^{n+p_q}=q^n\qquad\text{for every }n\geq N_q.
\]
Consequently the product is determined, for every \(q\), either by the exact
value of \(|R_q|\) when \(|R_q|<N_q\), or by the residue of \(|R_q|\) modulo
\(p_q\) together with the condition \(|R_q|\geq N_q\).  Exact bounded
cardinalities, lower bounds by fixed constants, and modular cardinalities are
expressible in \(\CMSO\).  A finite disjunction over the resulting count
profiles whose monoid product lies in \(A\) completes the construction.
\end{proof}

\begin{proof}[Proof of \cref{thm:main}]
Let \(\mathcal L=\mathcal P\cap\mathcal C_2\), where \(\mathcal P\) is
globally VR-recognizable.  By the labelled version of
\cref{prop:prime-definability}, every VR-recognizable language over a fixed
finite vertex-label alphabet is \(\CMSO\)-definable on prime graphs of
rank-width at most two.  Apply \cref{prop:split-lift}.  Thus \(\mathcal P\) is
\(\CMSO\)-definable on \(\mathcal C_2\), and hence so is \(\mathcal L\).
Conversely, if a \(\CMSO\) sentence \(\varphi\) defines \(\mathcal L\) on
\(\mathcal C_2\), Courcelle's theorem \cite{Courcelle90} makes its global
model class VR-recognizable, and its intersection with \(\mathcal C_2\) is
exactly \(\mathcal L\).
\end{proof}

\section{Conclusion}

Combining the preceding sections proves the recognizable-to-definable
direction of \cref{thm:main}.  The reverse direction is Courcelle's theorem,
and therefore the two notions coincide on graphs of rank-width at most two.
The point of the argument is not that such graphs have bounded linear
clique-width---they do not---but that their non-sequential rank-two cuts can
be organized into a definable laminar skeleton whose individual port torsos
have a uniform linear bound.  Recognizable properties can consequently be
evaluated locally and then composed through finite boundary states.

The proof also identifies the structural obstacle to extending the theorem
by the same route.  At rank two, compression reduces the flower-free pieces
to a 3-connected connectivity function of branch-width three, where the
Hall--Oxley--Semple--Whittle display theorem is available.  A higher-rank
extension would require an analogue strong enough to display successive
boundary blocks, or a different mechanism yielding bounded-width local
torsos.  The canonical-core and coherent-frame parts of the argument are
designed so that such a replacement could be inserted without changing the
finite-state evaluation.

\end{document}